\documentclass[11pt,letterpaper]{amsart}
 
  \usepackage{amsmath,amssymb,amsfonts,amsthm,amscd,textcomp,times,booktabs,mathrsfs}
   \usepackage{graphicx}
   \usepackage{hyperref}
    \usepackage{braket}
    \usepackage{tabularx}
  \usepackage{color,float}
  \usepackage{bm}
\usepackage[all]{xy}
 \usepackage{wrapfig}
 \usepackage[alphabetic]{amsrefs}
 \usepackage{url}
 \usepackage{pdflscape}

\DefineSimpleKey{bib}{myurl}

\newcommand\myurl[1]{\url{#1}}

\BibSpec{webpage}{%
  +{}{\PrintAuthors} {author}
  +{,}{ \textit} {title}
  +{}{ \parenthesize} {date}
  +{,}{ \myurl} {myurl}
  +{,}{ } {note}
  +{.}{ } {transition}
}

\newtheorem{theorem}{Theorem}

\newtheorem{claim}{Claim}
\newtheorem{proposition}[theorem]{Proposition}
\newtheorem{lemma}[theorem]{Lemma}

\newtheorem{corollary}[theorem]{Corollary}
\newtheorem{remark}[theorem]{Remark}

\numberwithin{equation}{section}
\numberwithin{theorem}{section}

\newcommand{\Z}{\ensuremath{\mathbb{Z}}}

\newcommand{\R}{\ensuremath{\mathbb{R}}}
\newcommand{\C}{\ensuremath{\mathbb{C}}}

\newcommand{\GL}{\ensuremath{\mathrm{GL}}}
\newcommand{\SL}{\ensuremath{\mathrm{SL}}}

\newcommand{\Bl}{\ensuremath{\mathrm{Bl}}}

\newcommand{\din}{\rotatebox{90}{\ensuremath{\in}}}

\newcommand{\reduce}[1]{\scalebox{1}{\ensuremath{#1}}}

\newcommand{\Vol}{\ensuremath{\mathrm{Vol}}}
\newcommand{\bracket}[1]{\ensuremath{\langle #1 \rangle}}

\DeclareMathOperator{\Aut}{Aut}

\DeclareMathOperator*{\smallprod}{\reduce{\prod}}

\DeclareMathOperator*{\smallcup}{\reduce{\bigcup}}

\DeclareMathOperator{\Ric}{Ric}

\DeclareMathOperator{\Pic}{Pic}

\def\bs{\boldsymbol}

\def\D{\Delta}

\def\O{\mathcal O}
\def\P{\mathbb{P}}

\begin{document}
%%%%%%%%%%%%%%%%%%%%%%%%%%%%%%%%%%%%%%%%%%%%%%%%%%%%
\title[CKE Projective Bundles]{Families of projective bundles that admit coupled K\"ahler-Einstein metrics but no K\"ahler-Einstein metrics}

\author{Naoto Yotsutani}
\address{Department of Mathematics, Faculty of Science, Shizuoka University, $836$ Ohya, Suruga-ku, Shizuoka-shi, Shizuoka, $422$-$8529$, Japan}
\email{yotsutani.naoto@shizuoka.ac.jp}
\address{IMAG, Universit\'{e} de Montpellier, $499$-$554$ Rue du Truel $9$, $34090$ Montpellier, France}
\email{naoto.yotsutani@umontpellier.fr}

\makeatletter
\@namedef{subjclassname@2020}{%
  \textup{2020} Mathematics Subject Classification}
\makeatother

 \dedicatory{}
%%%%%%%%%%%%%%%%%%%%%%%%%%%%%%%%%%%%%%%%%%%%%%%%%%%%%%
\date{\today}

\maketitle
%%%%%%%%%%%%%%%%%%%%%%%%%%%%%%%%%%%%%%%%%%%%%%%%%%

\noindent{\bfseries Abstract.}
We construct two families of projective bundles that admit two-coupled K\"ahler-Einstein (cKE) metrics but do not ordinary K\"ahler-Einstein metrics in
sufficiently high dimensions. We also prove that no such examples arise among toric Fano threefolds, for any decomposition of $c_1(X)$, when
considering the natural projection onto $\P^1$.

%%%%%%%%%%%%%% Sec 1 %%%%%%%%%%%%%%%%%%%%%%%%
\section{Introduction}
Let $X$ be a Fano manifold. A K\"ahler metric $\omega\in c_1(X)$ is called {\emph{K\"ahler-Einstein}} (KE) if
\[
\Ric(\omega)=\omega.
\]
Following the resolution of the Yau-Tian-Donaldson conjecture, the existence of a K\"ahler-Einstein metric on $X$ is equivalent to K-polystability.
In particular, many Fano manifolds do not admit K\"ahler-Einstein metrics due to algebro-geometric instability or the non-reductivity of their automorphism groups $\Aut(X)$.
It is therefore natural to consider broader notions (than KE metrics) of canonical K\"ahler metrics that still reflect the geometry of the first Chern class.

One such notion is that of a {\emph{coupled K\"ahler-Einstein}} (cKE) metric, introduced by Hultgren and Witt Nystr\"om \cite{HWN19} and further developed in subsequent works \cite{H19, FH24}. Given a decomposition
\[
c_1(X)=\alpha_1+\alpha_2,
\]
a pair of K\"ahler metrics $(\omega_1, \omega_2)$, with $[\omega_i]=\alpha_i$, is called a {\emph{two-coupled K\"ahler-Einstein metric}}
if it satisfies the coupled system
\[
\Ric(\omega_1)=\Ric(\omega_2)=\omega_1+\omega_2.
\]
More generally, given a decomposition 
\begin{equation}\label{eq:decomp}
c_1(X)=\alpha_1+\dots + \alpha_\ell
\end{equation}
into K\"ahler classes, one may seek an $\ell$-tuple of K\"ahler metrics $(\omega_1, \dots, \omega_\ell)$ satisfying a corresponding 
system of coupled Monge-Amp\`ere equations. 
When $\ell=1$, this reduces to the ordinary K\"ahler-Einstein problem, whereas for $\ell \geq 2$, genuinely new phenomena arise.
In particular, there exist Fano manifolds that do not admit KE metrics but do admit cKE metrics.
The uniqueness theorem for cKE metrics was established in \cite{HWN19}. It may be viewed as a natural generalization of the classical Band--Mabuchi uniqueness theorem for ordinary KE metrics.

In the case of a negative canonical class (i.e., $c_1(X)<0$), the decomposition $\eqref{eq:decomp}$ is assumed to be fixed.
Then there exists a unique cKE tuple $(\omega_1, \dots, \omega_\ell)$ in the prescribed K\"ahler classes. The proof relies on the strict convexity of the coupled Ding functional.

In the Fano case (i.e., $c_1(X)>0$), the decomposition $\eqref{eq:decomp}$ is again fixed. 
If a cKE metric exists, then it is unique up to the identity component $\Aut^0(X)$ of the automorphism group. More precisely, if $(\omega_1, \ldots, \omega_\ell)$ and $(\omega'_1, \ldots, \omega'_\ell)$ are two cKE tuples corresponding to the same decomposition of $c_1(X)$, 
then there exists $g\in \Aut^0(X)$ such that 
\[
g^*\omega_i=\omega'_i,  \qquad i=1,\dots, \ell.
\]
In particular, if $\Aut^0(X)=\set{1}$, then the cKE metric is unique.

In an earlier version of the present paper \cite{Y26}, the author considered certain projective bundles and proved that they do not admit ordinary KE metrics, while all of them admit two-coupled KE metrics in dimensions up to six.
More precisely, these projective bundles are
\begin{equation}\label{eq:ProjBdls}
X_k=\P_{\P^k\times \P^{k-1}}(\O\oplus \O(-1,1)), \qquad 
Y_k=\P_{\P^{k+1}\times \P^{k-1}}(\O\oplus \O(-1,1)),
\end{equation} 
for every integer $k\geq 2$, with $\dim X_k=2k$, $\dim Y_k=2k+1$, respectively. 
Our main interest lies in manifolds that
\begin{itemize}
\item do {\emph{not}} admit ordinary K\"ahler-Einstein metrics, but
\item {\emph{do}} admit coupled K\"ahler-Einstein metrics for a suitable decomposition of $c_1(X)$.
\end{itemize}
Based on the proof that $X_3$ and $Y_2$ satisfy these properties, the author conjectured that the same should hold for every integer $k\geq 2$ \cite[Conjecture 1.3]{Y26}. 
The purpose of the present paper is to prove this conjecture for all sufficiently large $k$. 
%Main Thm
\begin{theorem}\label{thm:main}
For $k\geq 2$, let $X_k$ and $Y_k$ be the families of projective bundles defined in $\eqref{eq:ProjBdls}$.
Then $X_k$ and $Y_k$ do not admit KE metrics, altough their automorphism groups are reductive.
Moreover, there exist decompositions
\begin{equation}\label{eq:decomp}
c_1(X_k)=\alpha_{1,k}+\alpha_{2,k}, \qquad c_1(Y_k)=\beta_{1,k}+\beta_{2,k}
\end{equation}
such that $(\alpha_{1,k}, \alpha_{2,k})$ and $(\beta_{1,k}, \beta_{2,k})$ admit two-coupled KE metrics for all sufficiently large $k$.
The statement also holds for $k=2$ (that is, for $X_2$ and $Y_2$) and for $X_3$.
\end{theorem}
As we explain below, the existence of cKE metrics for the family $X_k$ follows from Delcroix's work \cite[Theorem 1.4]{Del22}.
Moreover, $X_2$ is the toric Fano fourfold $D_{19}$ in Batyrev's notation \cite{Bat99}, and hence the result also follows from Hultgren's theorem
\cite[Theorem 2]{H19}. The existence of cKE metrics for the family $Y_k$ is established in Section \ref{sec:main}.
For the cases $Y_2$ and $X_3$,  we refer the reader to Proposition 3.2 and Theorem 3.4 of \cite{Y26}.

%Rem
\begin{remark}\rm
To prove the existence of cKE metrics for the families $X_k$ and $Y_k$ for arbitrary $k$, we need to take the limit as $k\to \infty$ in our argument.
See Sections \ref{sec:horospherical}--\ref{sec:toric} for further details.
It does not seem straightforward to determine an explicit $k_0$ such that $X_k$ and $Y_k$ admit suitable decompositions $\eqref{eq:decomp}$ for every integer $k\geq k_0$.
\end{remark}
Here we briefly explain the relationship between Theorem \ref{thm:main} and \cite[Theorem 1.4]{Del22}.
See Lemma \ref{lem:BlowUp} for more details. 

Let $V=A\oplus B$ be a direct sum decomposition of complex vector spaces,  where $A=\C^p$, $B=\C^q$, $n=p+q$.
Consider the rational map
\[
\pi: \P(V) \dasharrow \P(A)\times \P(B), \qquad [a+b]\longmapsto ([a],[b]).
\] 
This map is undefined precisely along the disjoint union $\P(A)\cup\P(B)$. It therefore extends to a morphism
\[
\widetilde{\pi}:\Bl_{\P(A), \P(B)}(\P(V)) \longrightarrow \P(A)\times \P(B).
\]
The fiber over a point $([a],[b])$ is the projective line $\P(\mathrm{span}_\C\set{a,b})$. 
Hence $\widetilde{\pi}$ is a $\P^1$-bundle over $\P(A)\times \P(B)$. More precisely, its fiber is $\P(\O(-1,0)\oplus \O(0,-1))$. Thus we have
\[
\Bl_{\P(A), \P(B)}(\P(V)) \cong \P_{\P^{p-1}\times \P^{q-1}}(\O(-1,0)\oplus \O(0,-1))
\]
and, after tensoring by $\O(1,0)$,
\[
\P_{\P^{p-1}\times \P^{q-1}}(\O(-1,0)\oplus \O(0,-1)) \cong \P_{\P^{p-1}\times \P^{q-1}}(\O\oplus \O(-1,1)).
\] 
Taking $p=k$ and $q=k+1$, so that $n=2k+1$, we obtain
\[
\Bl_{\P^{k-1},\P^{k}}(\P^{2k})\cong \P_{\P^{k}\times \P^{k-1}}(\O\oplus \O(-1,1))=:X_k.
\]
This family was essentially considered in \cite[Section 2.6]{Del22}. We therefore focus on the second family $Y_k$ in the present paper.
We prove Theorem \ref{thm:main} in two different ways:
\begin{itemize}
\item From the perspective of horospherical geometry, following \cite{Del22}, we prove the existence of cKE metrics by a continuity argument (Section \ref{sec:horospherical}).
\item From the perspective of toric geometry, we apply Hultgren's toric criterion for cKE metrics to verify the assertion.
\end{itemize}
Proposition \ref{prop:ProjBdls} shows that $\Aut^0(X_k)$ and $\Aut^0(Y_k)$ are reductive for every integer $k\geq 2$.
In the final section, we investigate the existence problem of cKE metrics on toric Fano threefolds.

\vskip 7pt

\noindent{\bfseries Acknowledgements.}
I would like to thank Professors Hiroshi Sato, Benjamin Nill and Thibaut Delcroix for their helpful comments on an earlier version of this manuscript. 
In particular, I am grateful to Professor Nill for kindly suggesting a practical approach to \cite[Conjecture 1.3]{Y26} from a toric-geometric viewpoint. 
I also would like to thank the anonymous referee for bringing to my attention the closely related work of Delcroix \cite{Del22}.

This work was partially supported by JSPS KAKENHI Grants JP$22$K$03316$, $24$KK$0252$, and by NSFC Grant $12571058$.

%%%%%%%%%%%Sec 2
\section{Reductivity of Automorphism Groups of toric Fano varieties}\label{sec:automorphism}
Let $G$ be a connected linear algebraic group over $\C$, and let $G_u$ denote its unipotent radical (the maximal connected normal unipotent subgroup). 
Recall that $G$ is {\emph{reductive}} if and only if $G_u$ is trivial.

We begin by reviewing known results concerning the reductivity of automorphism groups of  Fano varieties.

Let $X$ be a Fano manifold, and denote by $\Aut^0(X)$ the connected component of the identity in $\Aut(X)$.
By Matsushima's theorem \cite{Mat57}, if $X$ admits a K\"ahler-Einstein metric in $c_1(X)$, then the Lie algebra of holomorphic vector fields on $X$ is reductive;
equivalently, $\Aut^0(X)$ is reductive. This result was later generalized to the coupled setting: 
if $X$ admits a coupled K\"ahler-Einstein (cKE) metric, then $\Aut^0(X)$ is reductive (\cite[Corollary $1.6$]{HWN19}).

Futaki \cite{Fut83} introduced a Lie algebra character, now known as the {\em{Futaki invariant}}, which vanishes identically 
if $X$ admits a K\"ahler-Einstein metric. 
In general, the vanishing of the Futaki invariant is a necessary condition for the existence of KE metrics.

In the toric case, Wang and Zhu \cite{WZ04} proved that this condition is also sufficient: 
a toric Fano manifold admits a KE metric {\emph{if and only if}} its Futaki invariant vanishes.
Moreover, Mabuchi \cite{Mab87} showed that for a toric Fano manifold the Futaki invariant can be computed as the barycenter of the associated reflexive Delzant polytope. 

Combining these results, one sees that if the barycenter of the reflexive Delzant polytope $P$ associated with a toric Fano manifold $X$ is zero, then $X$ admits a KE metric and hence $\Aut^0(X)$ is reductive.
From a more structural viewpoint, the set of roots (known as {\em Demazure's roots}) of a complete toric variety $X$ plays a crucial role in determining $\Aut(X)$. 
From this perspective, Nill \cite{Ni06} proved that the automorphism group of a Gorenstein toric Fano variety is reductive whenever the barycenter of the associated (not necessarily Delzant) reflexive polytope is zero, regardless of whether $X$ admits a KE metric.
%Thm
\begin{theorem}[\cite{Ni06}]
Let $X$ be a Gorenstein toric Fano variety, and $P \subset M_{\R}$ its associated reflexive polytope.
If the barycenter of $P$ is $\bf 0$, then $\mathrm{Aut}^0(X)$ is reductive. 
\end{theorem}
However, the converse does not hold in general, even in the smooth toric Fano case. %where $X$ is toric Fano {\emph{manifolds}}. 
In particular, we have the following result in low dimensions.
%Prop
\begin{proposition}[\cite{Y17}, Proposition 4.3]\label{prop:Aut}

\begin{enumerate}
\item In dimension three, there exists a unique non-KE toric Fano $3$-fold whose automorphism group is reductive, namely the  
variety of type $\mathcal {F}_2$ in Batyrev's notation \cite{Bat99}.
\item In dimension four, the toric Fano $4$-folds $D_{19}, H_{10}, J_2$ and $Q_{17}$ have reductive automorphism groups, 
although none of them admits a KE metric in $c_1(X)$.
\end{enumerate}
\end{proposition}
\begin{proof}
See, \cite[Section 4.1]{Y17}
\end{proof}
%Rem
\begin{remark}\rm
In \cite[Corollary 4]{H19}, Hultgren proved that $D_{19}=\P_{\P^1\times \P^2}(\mathcal O\oplus \mathcal O(-1,1))$ is an example of toric Fano manifolds that does not admit
a KE metric, but for which there exists a decomposition $c_1(X)=\alpha_1+\alpha_2$ (with {\emph{irrational}} coefficients) admitting a cKE metric. 
\end{remark}
Motivated by Proposition \ref{prop:Aut}, one may ask whether the other candidates listed above, namely the toric Fano $3$-fold $\mathcal F_2$ and
the $4$-folds $H_{10}$, $J_2$, $Q_{17}$, admit cKE metrics. Our investigations suggest that, among toric Fano manifolds of dimension $\leq 4$,
$D_{19}$ is the unique example admitting a two coupled decomposition of $c_1(X)$ that yields a cKE metric while failing to admit a KE metric.

%sec2.1
\subsection{Demazure Roots and Reductivity}
For the reader's convenience, we briefly recall the notion of the {\emph{set of roots}} of a complete toric variety, introduced by Demazure.
See \cite[Section $3.4$]{Oda88} and \cite{Ni06} for details.

Let $\Sigma$ be the complete fan associated with a (smooth) toric Fano variety $X$.
Recall that a fan $\Sigma$ in $N_\R$ is {\emph{complete}} if 
\[
\smallcup_{\sigma\in \Sigma}\sigma=N_\R.
\]
For each ray (i.e., one-dimensional cone) $\rho \in \Sigma(1)$, let $\bs v_\rho$ denote its primitive generator. 
The {\emph{set of roots}} of $\Sigma$ is defined by 
\begin{align}\label{eq:Root}
\mathcal R&:=\Set{m\in M  |  {}^\exists \rho\in\Sigma(1) ~\text{~ s.t. ~}~ \bracket{\bs v_\rho,m}=-1,  \bracket{\bs v_{\rho'}, m} \geq 0 ~ \text{ for}~\,  
{}^\forall \rho'\in \Sigma(1)\setminus \set{\rho}}.
 \end{align}
Define
\[
\mathcal S:=\mathcal R \cap \left (  -\mathcal R \right ) =\Set{m \in \mathcal R | -m \in \mathcal R},
\]
which is called the {\emph{set of semisimple roots}}, and
\[
\mathcal U:=\mathcal R \setminus \mathcal S=\Set{m \in \mathcal R | -m \notin \mathcal R}
\]
which is called the {\emph{set of unipotent roots}}. 
For each $m\in \mathcal R$, there exists a corresponding one-parameter subgroup 
\[
\chi_m: \C \to \Aut (X).
\]
Demazure's structure Theorem (\cite[p. $140$]{Oda88}) states that the unipotent radical $G_u$ of $\Aut^0(X)$ is isomorphic to the product 
\[
\smallprod_{m\in \mathcal U}\chi_m(\C).
\]
Moreover, there exists a reductive algebraic subgroup $G_s$, containing $(\C^\times)^n$ as a maximal algebraic torus and having root system $\mathcal S$, such that
\[
\Aut^0(X)=G_u \rtimes G_s.
\]
A complete fan $\Sigma$ is called {\emph{semisimple}} if $\mathcal R=\mathcal S$. 
The following criterion is well known.
%Lem
\begin{lemma}\label{lem:reductive}
The automorphism group $\Aut^0(X)$ of a complete toric variety $X$ is reductive if and only if the associated complete fan $\Sigma$ is semisimple.
\end{lemma}
In the Gorenstein toric Fano case, one has the following useful sufficient condition.
%Lem
\begin{lemma}[\cite{Ni06}, Theorem 5.2]\label{lem:suff}
Let $X$ be a Gorenstein toric Fano variety with the fan $\Sigma$. If
\[
\sum_{\rho\in \Sigma(1)}\bs v_\rho = \bs 0,
\]
then $\Aut^0(X)$ is reductive.
\end{lemma}

%Sec2.2
\subsection{Arbitrary dimensional case}\label{sec:arbitrary}
For any integer $k\geq 2$, let $\bs e_i$ $(i=1, \dots, 2k)$ be the standard basis of $N\cong \Z^{2k}$. We consider the complete fan in $N_\R$ whose rays are generated by
\[
\bs v_i=\bs e_i ~~ (i=1, \dots, 2k), \quad \bs v_{2k+1}=\bs v_{2k}-\sum_{i=1}^k \bs v_i, \quad \bs v_{2k+2}=-\sum_{i=k+1}^{2k}\bs v_i, \quad \bs v_{2k+3}=-\bs v_{2k}.
\]
These $\bs v_i$ satisfy the relations
\begin{equation}\label{eq:PR_Gen}
\sum_{i=1}^k \bs v_i+\bs v_{2k+1}=\bs v_{2k}, \qquad \sum_{i=k+1}^{2k-1} \bs v_i+\bs v_{2k+2}=\bs v_{2k+3}, \qquad \bs v_{2k}+ \bs v_{2k+3}=\bs 0.
\end{equation}
This yields that the associated toric variety is
\[
X_k=\P_{\P^k\times \P^{k-1}}(\mathcal O\oplus \mathcal O(-1,1))
\]
with $\dim X_k=2k$. 

Similarly we can construct a $(2k+1)$-dimensional toric variety 
\[
Y_k=\P_{\P^{k+1}\times \P^{k-1}}(\mathcal O\oplus \mathcal O(-1,1))
\]
associated with the fan $\Sigma_k$ whose rays are generated by
%\begin{align*}
\[
\bs v_i=\bs e_i ~~  (i=1, \dots, 2k+1), \quad \bs v_{2k+2}=\bs v_{2k+1}-\sum_{i=1}^{k+1} \bs v_i, \quad
\bs v_{2k+3}=-\sum_{i=k+2}^{2k+1}\bs v_i, \quad 
\bs v_{2k+4}=-\bs v_{2k+1}.
\]
%\end{align*}
%Prop
\begin{proposition}\label{prop:ProjBdls}
The projective bundles $X_k$ and $Y_k$ do not admit K\"ahler-Eintein metrics in their first Chern classes. However, $\Aut^0(X_k)$ and $\Aut^0(Y_k)$ are reductive.
\end{proposition}
\begin{proof}
Clearly, neither $X_k$ nor $Y_k$ admits a KE metric. Thus it suffices to show that the connected components of their automorphism groups are reductive.

From $\eqref{eq:PR_Gen}$, we obtain
\begin{align*}
\sum_{i=1}^{2k+3}\bs v_i&=\bigl( \sum_{i=1}^{k} \bs v_i+\bs v_{2k+1} \bigr)+\bigl( \sum_{i=k+1}^{2k-1} \bs v_i+\bs v_{2k+2} \bigr)+\bs v_{2k}+\bs v_{2k+3}\\
&=2\bigl(\bs v_{2k}+\bs v_{2k+3} \bigr)=\bs 0.
\end{align*}
Thus $\Aut^0(X_k)$ is reductive by Lemma \ref{lem:suff}. The case of $\Aut^0(Y_k)$ is similar.
\end{proof}

%Sec3
\section{Proofs of the main theorem}\label{sec:main}
In this section, we prove Theorem \ref{thm:main} using two different approaches. The first is based on horospherical geometry (Section \ref{sec:horospherical}), while the second uses toric geometry (Section \ref{sec:toric}).
To apply the argument of \cite[Theorem 1.4]{Del22} to our setting, we first establish the following isomorphism:
\begin{equation}\label{isom:ProjBdl}
\mathrm{Bl}_{\P^{p-1},\P^{q-1}}(\P^{n-1})\cong \P_{\P^{p-1}\times \P^{q-1}}(\O\oplus \O(-1,1)).
\end{equation}

%Sec3.1
\subsection{The projective bundle description of blow-up manifolds}
Let $A$ and $B$ be complex vector spaces of dimensions $p$ and $q$, respectively, such that $A\cap B=\set{\bs 0}$.
Let $V:=A\oplus B$, $n:=p+q$. Then $\P(V)=\P^{n-1}$. Let
\[
L_A:=\P(A)\cong \P^{p-1}, \qquad L_B:=\P(B)\cong \P^{q-1}.
\]
Since $L_A\cap L_B=\emptyset$, we define $X_{n,p}:=\Bl_{L_A, L_B}(\P(V))$. We have the following isomorphism.
%Lem
\begin{lemma}\label{lem:BlowUp}
$X_{n,p}\cong \P_{\P(A)\times \P(B)}(\O\oplus \O(-1,1))$
\end{lemma}
\begin{proof}
For nonzero vectors $a\in A$ and $b\in B$, the point $[a+b]\in \P(V)$ uniquely determines the point $([a],[b])\in \P(A)\times \P(B)$.
Thus, we obtain a rational map
\begin{equation*}\label{map:pi1}
\xymatrix@R=-1ex{
\pi: \P(V) \ar@{-->}[r]& \P(A)\times \P(B)\\
\qquad \!\!\! \din &\!\! \din\\
[a+b] \ar@{|->}[r]& ([a],[b]).
 }
\end{equation*}
The map $\pi$ is undefined precisely on the indeterminacy locus $L_A\cup L_B$.
Since $L_A$ and $L_B$ are disjoint, blowing up their union resolves the indeterminacy of $\pi$ and yields a morphism
\begin{equation}\label{diagram:Extend}
\xymatrix{
\widetilde{\pi}: X_{n,p}=\Bl_{L_A\cup L_B}(\P(V)) \ar@{-->}[d]\ar[r]  &{\P(A)\times \P(B)} \\
\P(V) \ar@{-->}[ur]_-{\pi} & 
}
\end{equation}
extending $\pi$, as illustrated by the diagram $\eqref{diagram:Extend}$.
To determine the fibers of $\widetilde{\pi}$, fix apoint $([a],[b])\in \P(A)\times \P(B)$, and choose representatives $a\in A$ and $b\in B$.
Then %one can see that the fiber is given by
\begin{align*}
\widetilde{\pi}^{-1}([a],[b])&=\Set{[\lambda a+\mu b]\in \P(V)|x=[\lambda : \mu ]\in \P^1} \\
&=\P\bigl(\mathrm{span}_\C\set{a,b}\bigr)\cong\P^1.
\end{align*} 
Therefore, $\widetilde{\pi}: X_{n,p} \to \P(A)\times \P(B)$ is a $\P^1$-bundle.

We now identify the vector bundle defining this $\P^1$-bundle. Set $Z:=\P(A)\times \P(B)$, and let
\[
S_A:=\O_{\P(A)}(-1)\cong \O_Z(-1,0), \qquad  S_B:=\O_{\P(B)}(-1)\cong \O_Z(0,-1)
\]
be the tautological line bundles. For each point $([a],[b])\in Z$, the fiber of $S_A\oplus S_B$ is
\[
(S_A)_{[a]}\oplus (S_B)_{[b]}=\C a\oplus \C b=\mathrm{span}_\C\set{a,b}.
\]
It follows that $X_{n,p}\cong \P_Z(S_A\oplus S_B)$. Since projectivization is unchanged after tensoring by a line bundle,
we have $\P(E)\cong \P(E\otimes M)$. Taking $M=S_B^{-1}=\O_{\P(B)}(1)\cong \O_Z(0,1)$, we obtain
\[
(S_A\oplus S_B)\otimes M=\O_Z\oplus(S_A\otimes S_B^{-1})\cong \O_Z\oplus \O_Z(-1,1).
\]
This proves the desired description $X_{n,p}\cong \P_{\P(A)\times \P(B)}(\O\oplus \O(-1,1))$.
\end{proof}
Taking $p=k$, $q=k+2$, so that $n=2k+2$ in $\eqref{isom:ProjBdl}$, we obtain
\[
\Bl_{\P^{k-1},\P^{k+1}}(\P^{2k+1})\cong \P_{\P^{k-1}\times \P^{k+1}}(\O\oplus \O(-1,1))=:Y_k.
\]
In the following section, we prove Theorem \ref{thm:main} from the perspective of horospherical geometry.

%Sec3.2
\subsection{Horospherical geometric viewpoint}\label{sec:horospherical}
Throughout  this subsection, we use the same notation in \cite[Section 2]{Del22}.

Consider the special subgroup of the block-diagonal group $\GL_p\times \GL_{q}$ given by
\[
G:=\Set{(A,B)\in \GL_p(\C)\times \GL_{q}(\C)| \det A \cdot \det B=1}, \qquad n=p+q.
\]
We regard $G$ as a subgroup of $\GL_n(\C)$ via
$(A,B)\longmapsto \begin{pmatrix}
A & O \\
O & B
\end{pmatrix}$. 
It was proved in \cite[Section 2.2]{Del22} that, under the action of $G$, the manifold 
$X_{n,p}:=\mathrm{Bl}_{\P^{p-1},\P^{q-1}}(\P^{n-1})$ is horospherical.
In particular, its horospherical structure of $X_{n,p}$ is well understood, and the relevant combinatorial data are explicitly described there.

Let $D_{+}$, $D_-$ be the $G$-invariant divisors of $X_{n,p}$, while $D_{\alpha_{2, 1}}$, $D_{\alpha_{p+2, p+1}}$ be the two color divisors
(that is, the $B$-stable but not $G$-stable prime divisors in the spherical-variety terminology, where $B$ is the Borel subgroup of $G$). 
In the geometric picture of a blow-up of $\P^{n-1}$, these colors correspond to the two distinguished divisors related to the two projective subspaces $\P^{p-1}$ and $\P^{q-1}$.  
The Picard group is generated by these four $G$-invariant divisors $D_{\alpha_{2, 1}}$, $D_{\alpha_{p+2, p+1}}$, $D_{+}$ and $D_-$ with one linear relation
\[
D_{+} + D_{\alpha_{p+2, p+1}}=D_- + D_{\alpha_{2, 1}}.
\]
For $a_c$, $a_+$, $a_{-}\in \R$, one can parameterize all the {\emph{special}} $\R$-divisors as
\[
D(a_c, a_{+}, a_{-}):=a_c(D_{\alpha_{2, 1}}+D_{\alpha_{p+2, p+1}})+a_{+}D_{+}+a_{-}D_{-}.
\]
Here, by a special $\R$-divisor, we mean a $G$-invariant $\R$-divisor of the above form.
Note that $D(a_c, a_{+}, a_{-})$ is ample if and only if $a_{+}<a_c$, $a_{-}<a_c$ and $a_{+}+a_{-}>0$.
%%%%%%%%%%%%%
The Duistermaat-Heckman barycenter is a weighted average associated with the moment polytope (or, more generally, the relevant spherical data) of the polarized variety. In this setting, the cKE existence criterion can be expressed as a condition on this barycenter.
The Duistermaat-Heckman barycenter $\mathrm{Bar}(a_c, a_{+}, a_{-})$ corresponding to the ample divisor $D(a_c, a_{+}, a_{-})$ is calculated in \cite[p. $2089$]{Del22}.

We now specialize to $p=k$, $q=k+2$ and $n=2k+2$.
Set
\[
a_c'=k+1-a_c, \qquad a_{+}'=\frac{1}{2}-a_{+}, \qquad a_{-}'=\frac{3}{2}-a_{-}.
\] 
Then the $G$-invariant divisors $D(a_c, a_{+}, a_{-})$ and $D(a'_c, a'_{+}, a'_{-})$ are both ample if and only if
\begin{equation}\label{ineq:ample}
a_{+}<a_c<a_{+}+k+\frac{1}{2}, \qquad a_{-}<a_c<a_{-}+k-\frac{1}{2}, \qquad 0<a_{+}+a_{-}<2.
\end{equation}
By \cite[Theorem 1.2]{DH21}, the existence problem for cKE metrics on $X_{n,p}\cong Y_k$ is reduced to finding $(a_c, a_{+}, a_{-})$ satisfying the barycentric condition
\begin{equation}\label{eq:Bary}  
\frac{\displaystyle 
\int_{-a_{+}}^{a_{-}}t(t+a_c)^{k+1}(a_c-t)^{k-1}dt}{\displaystyle 
\int_{-a_{+}}^{a_{-}}(t+a_c)^{k+1}(a_c-t)^{k-1}dt}+
\frac{\displaystyle 
\int_{-a'_{+}}^{a'_{-}}t(t+a'_c)^{k+1}(a'_c-t)^{k-1}dt}{\displaystyle 
\int_{-a'_{+}}^{a'_{-}}(t+a'_c)^{k+1}(a'_c-t)^{k-1}dt}
-\frac{1}{2}=0
\end{equation}
subject to the ampleness condition $\eqref{ineq:ample}$. For convenience, let $Q_1(a_c, a_{+}, a_{-})$ and $Q_2(a_c, a_{+}, a_{-})$
denote the first and second terms, respectively, in $\eqref{eq:Bary}$, and define
\[
C(a_c, a_{+}, a_{-}):=Q_1(a_c, a_{+}, a_{-})+Q_2(a_c, a_{+}, a_{-})-\frac{1}{2}.
\]
Thus, it suffices to find a triple $(a_c, a_{+}, a_{-})$ satisfying $\eqref{ineq:ample}$ such that $C(a_c, a_{+}, a_{-})=0$.
We first prove the following.
%Lem1
\begin{lemma}\label{lem:positive}
For $(a_c, a_{+}, a_{-})=\bigl(\frac{n}{4}, \frac{1}{4}, \frac{3}{4} \bigr)$, we have $C(a_c, a_{+}, a_{-})>0$.
\end{lemma}
\begin{proof}
If we choose $a_c=\frac{n}{4}=\frac{k+1}{2}$, $a_{+}=\frac{1}{4}$, $a_{-}=\frac{1}{4}$, then 
\[
(a'_c, a'_{+}, a'_{-})=\left(\frac{k+1}{2}, \frac{1}{4},  \frac{1}{4} \right)
\]
and this symmetric point satisfies $\eqref{ineq:ample}$. Thus,
\[
C\left(\frac{n}{4}, \frac{1}{4}, \frac{3}{4} \right)=2Q_1\left(\frac{n}{4}, \frac{1}{4}, \frac{3}{4} \right)-\frac{1}{2},
\]
It is therefore enough to show that $Q_1\left(\frac{n}{4}, \frac{1}{4}, \frac{3}{4} \right)> \frac{1}{4}$.
Recall that
\[
Q_1\left(\frac{k+1}{2}, \frac{1}{4}, \frac{3}{4} \right)=
\frac{\int_{-\frac{1}{4}}^{\frac{3}{4}}t\left(t+\frac{k+1}{2}\right)^{k+1}\left(\frac{k+1}{2}-t\right)^{k-1}dt}
{\int_{-\frac{1}{4}}^{\frac{3}{4}}\left(t+\frac{k+1}{2}\right)^{k+1}\left(\frac{k+1}{2}-t\right)^{k-1}dt}.
\]
Now we define a positive function $w(t)=\left(t+\frac{k+1}{2}\right)^{k+1}\left(\frac{k+1}{2}-t\right)^{k-1}$ on $\left[ -\frac{1}{4}, \frac{3}{4}\right]$. Then
\begin{equation}\label{eq:Q1}
Q_1\left(\frac{k+1}{2}, \frac{1}{4}, \frac{3}{4} \right)=
\frac{\int_{-\frac{1}{4}}^{\frac{3}{4}}\left(t-\frac{1}{4} \right)\left(t+\frac{k+1}{2} \right)w(t)\,dt}
{\int_{-\frac{1}{4}}^{\frac{3}{4}}\left(t+\frac{k+1}{2} \right)w(t)\,dt}.
\end{equation}
Using
\[
\left(t-\frac{1}{4} \right)\left(t+\frac{k+1}{2} \right)=\left(t-\frac{1}{4} \right)
\left\{ \left(t-\frac{1}{4} \right) +\frac{2k+3}{4}\right\},
\]
the numerator in $\eqref{eq:Q1}$ is 
\begin{equation}\label{eq:Q1'}
\int_{-\frac{1}{4}}^{\frac{3}{4}}\left(t-\frac{1}{4} \right)^2w(t)\, dt+\frac{2k+3}{4}\int_{-\frac{1}{4}}^{\frac{3}{4}}\left(t-\frac{1}{4} \right)w(t)\, dt.
\end{equation}
The first term is strictly positive. It therefore remains to show the second term in $\eqref{eq:Q1'}$ is positive.

Setting $s=t-\frac{1}{4}$, we obtain
\begin{align}
\begin{split}\label{eq:Q1''}
\int_{-\frac{1}{4}}^{\frac{3}{4}}\left(t-\frac{1}{4} \right)w(t)\, dt&=
\int_{-\frac{1}{2}}^{\frac{1}{2}}s\left(s+\frac{2k+3}{4} \right)^k\left(\frac{2k+1}{4}-s \right)^{k-1}ds\\
&=\frac{1}{2^{k+1}}\int_{-1}^1x\left(x+\frac{2k+3}{2} \right)^k\left(\frac{2k+1}{2}-x \right)^{k-1}dx,
\end{split}
\end{align}
where we put $x=2s$ in the second equality. Shifting $x\mapsto x+\frac{1}{2}$ if we need, one can see that $\eqref{eq:Q1''}$
corresponds to $I(2k+1,k)$ in \cite[Lemma 2.2]{Del22}. Since $I(2k+1,k)>0$, we conclude that
\[
C\left(\frac{k+1}{2}, \frac{1}{4}, \frac{3}{4} \right)=2Q_1\left(\frac{k+1}{2}, \frac{1}{4}, \frac{3}{4} \right)-\frac{1}{2}>0.
\]
This proves the lemma.
\end{proof}
Since the K\"ahler cone is connected and the functional $C(a_c,a_{+}, a_{-})$ is continuous, 
it remains to find a triple $(a_c,a_{+}, a_{-})$ satisfying $\eqref{ineq:ample}$ for which $C(a_c,a_{+}, a_{-})<0$. 
We now prove the following.
%Lem2
\begin{lemma}\label{lem:negative}
Set $a_c=k-2$, $a_{+}=\frac{1}{2}$ and $a_{-}=0$. Then this triple satisfies $\eqref{ineq:ample}$ for $k\geq 3$.
Moreover, $C(a_c, a_{+}, a_{-})<0$ for all sufficiently large $k$.
\end{lemma}
\begin{proof}
For $(a_c, a_{+}, a_{-})=\bigl(k-2, \frac{1}{2}, 0 \bigr)$, and $C(a_c, a_{+}, a_{-})=Q_1(a_c, a_{+}, a_{-})+Q_2(a_c, a_{+}, a_{-})-\frac{1}{2}$, one can conclude that
\[
\lim_{k\to \infty} Q_1(a_c, a_{+}, a_{-}) =-\frac{1}{4}
\]
as the same manner in \cite[p.2093]{Del22}. Namely, using
\[
\frac{(t+k-2)^{k+1}(k-2-t)^{k-1}}{(k-1)^{2k}}=\left(1+\frac{t-1}{k-1}\right)^2\left(1+\frac{t-1}{k-1}\right)^{k-1}\left(1-\frac{t+1}{k-1}\right)^{k-1},
\]
one can see that
\[
\lim_{k\to \infty} Q_1(k-2, \frac{1}{2}, 0)=\frac{\int_{-\frac{1}{2}}^0 te^2dt}{\int_{-\frac{1}{2}}^0 e^2dt}=-\frac{1}{4}.
\]
For $(a'_c, a'_{+}, a'_{-})=\left( 3,0, \frac{3}{2}\right)$, we need to prove
\begin{equation}\label{eq:limtQ2}
\lim_{k\to \infty}Q_2(a'_c, a'_{+}, a'_{-})=0.
\end{equation}
To this end, we use a Laplace estimate.
Let $w_k(t):=(t+3)^{k+1}(3-t)^{k-1}$. Then $w_k(t)$ is a positive function on $\left[0, \frac{3}{2}\right]$ and
\[
Q_2(a'_c, a'_{+}, a'_{-})=\frac{\int_0^{\frac{3}{2}}tw_k(t)\,dt}{\int_0^{\frac{3}{2}}w_k(t)\,dt}.
\]
Since 
\[
w_k(t)=(3-t)^{-2}(9-t^2)^{k+1}=(3-t)^{-2}\exp\left((k+1)\log (9-t^2) \right),
\]
we put $\phi(t):=\log (9-t^2)$. The direct computation shows that
\[
\phi'(t)=-\frac{2t}{9-t^2}, \qquad \phi''(t)=\frac{-2(t^2+9)}{(9-t^2)^2}.
\]
In particular, $\phi''(0)=-\frac{2}{9}<0$. Moreover, $\phi'(t)<0$ for $0<t<\frac{3}{2}$, so $\phi(t)$ is strictly decreasing.

Fix any $\delta>0$. For $t\in \left[ \delta, \frac{3}{2}\right]$ we see that
\[
\phi(t)\leq \phi(\delta)< \phi(0)=\phi_{\max}
\]
because $\phi$ is a decreasing function. Also we have
\[
\frac{1}{(3-t)^2}\leq \left( \frac{2}{3}\right)^2=\frac{4}{9}
\]
on $\left[ 0, \frac{3}{2}\right]$. Therefore, $w_k(t)\leq \frac{4}{9}\exp\bigl( (k+1)\phi (\delta)\bigr)$, and hence
\begin{equation}\label{ineq:Estim1}
\int_{\delta}^{\frac{3}{2}}w_k(t)\, dt\leq \frac{4}{9}\cdot \frac{3}{2}e^{(k+1)\phi(\delta)}=\frac{2}{3}e^{(k+1)\phi(\delta)}.
\end{equation}

On the other hand, fix some $A>0$ so that for sufficiently large $k$, we have $\frac{A}{\sqrt{k+1}}<\delta$.
Thus,
\[
\int_{0}^{\delta}w_k(t)\, dt\geq \int_{0}^{\frac{A}{\sqrt{k+1}}}w_k(t)\, dt.
\]
For sufficiently small $t$, the Taylor theorem gives that
\[
\phi(t)=\phi(0)-\frac{1}{9}t^2+o(t^4)
\]
with $\phi(0)=\log 9$. Thus there exists a constant $C_1>0$ such that
\[
\phi(t)\geq \phi(0)-C_1t^2, \qquad 0\leq t \leq \delta.
\]
Hence, for $0\leq t \leq \frac{A}{\sqrt{k+1}}$, we have
\[
(k+1)\phi(t) \geq (k+1)\phi(0)-C_1(k+1)t^2\geq (k+1)\phi(0)-C_1A^2,
\]
which implies that $e^{(k+1)\phi(t)}\geq e^{(k+1)\phi(0)-C_1A^2}$.
Also, on $0\leq t\leq \delta$, we have $\frac{1}{(3-t)^2}\geq \frac{1}{9}$. Thus,
\begin{equation}\label{ineq:wk}
w_k(t)=\frac{e^{(k+1)\phi(t)}}{(3-t)^2}\geq \frac{1}{9}e^{-C_1A^2}e^{(k+1)\phi(0)}.
\end{equation}
Integrating both sides of $\eqref{ineq:wk}$, we conclude that
\[
\int_{0}^{\delta}w_k(t)\, dt\geq \int_{0}^{\frac{A}{\sqrt{k+1}}}w_k(t)\, dt\geq \frac{1}{9}e^{-C_1A^2}e^{(k+1)\phi(0)}\cdot 
\frac{A}{\sqrt{k+1}}.
\]
Thus, setting $C:=\frac{A}{9}e^{-C_1A^2}$, we obtain
\begin{equation}\label{ineq:Estim2}
\int_{0}^{\delta}w_k(t)\, dt \geq C\cdot \frac{e^{(k+1)\phi(0)}}{\sqrt{k+1}}.
\end{equation}
Combining $\eqref{ineq:Estim1}$ and $\eqref{ineq:Estim2}$, we estimate
\[
\frac{\int_{\delta}^{\frac{3}{2}}w_k(t)\, dt}{\int_{0}^{\delta}w_k(t)\, dt}\leq \frac{\frac{2}{3}e^{(k+1)\phi(\delta)}}{C\cdot \frac{e^{(k+1)\phi(0)}}{\sqrt{k+1}}}
=\frac{2\sqrt{k+1}}{3C}e^{-(k+1)\bigl( \phi(0)-\phi(\delta)\bigr)}.
\]
Since $\phi(0)-\phi(\delta)>0$, the exponential decay dominates the factor $\sqrt{k+1}$, and therefore
\[
\lim_{k\to \infty} \frac{\int_{\delta}^{\frac{3}{2}}w_k(t)\, dt}{\int_{0}^{\delta}w_k(t)\, dt}=0.
\]
Consequently,
\[
0\leq \frac{\int_{\delta}^{\frac{3}{2}}w_k(t)\, dt}{\int_{0}^{\frac{3}{2}}w_k(t)\, dt}\leq 
\frac{\int_{\delta}^{\frac{3}{2}}w_k(t)\, dt}{\int_{0}^{\delta}w_k(t)\, dt},
\]
so we proved
\begin{equation}\label{eq:limit_Int}
\lim_{k\to \infty} \frac{\int_{\delta}^{\frac{3}{2}}w_k(t)\, dt}{\int_{0}^{\frac{3}{2}}w_k(t)\, dt}=0.
\end{equation}
Now we estimate the weighted average. Let $R_k:=Q_2(a'_c, a'_{+}, a'_{-})$.
Since
\begin{align*}
\int_0^\delta tw_k(t)\, dt &\leq %\int_0^\delta \delta w_k(t)\, dt=
\delta \int_0^\delta w_k(t)\, dt, \qquad \text{and} \qquad
\int_\delta^{\frac{3}{2}} tw_k(t)\, dt \leq \frac{3}{2} \int_\delta^{\frac{3}{2}} w_k(t)\, dt,
\end{align*}
we obtain
\begin{align*}
R_k&\leq \delta +\frac{3}{2}\cdot \frac{\int_\delta^{\frac{3}{2}} w_k(t)\, dt}{\int_{0}^{\frac{3}{2}}w_k(t)\, dt}
\end{align*}
By $\eqref{eq:limit_Int}$, $\lim_{k\to \infty}\sup R_k\leq \delta$. Since $\delta>0$ is arbitrary and $R_k\geq 0$, we conclude that
\[
\lim_{k\to \infty}R_k=0. 
\]
This proves $\eqref{eq:limtQ2}$.

Consequently, we proved that
\[
\lim_{k\to \infty}Q_1(a_c, a_{+}, a_{-})=-\frac{1}{4}, \qquad \lim_{k\to \infty}Q_2(a'_c, a'_{+}, a'_{-})=0
\]
for $C(a_c, a_{+}, a_{-})=Q_1(a_c, a_{+}, a_{-})+Q_2(a'_c, a'_{+}, a'_{-})$. Thus, $C(a_c, a_{+}, a_{-})<0$ for all sufficiently large $k$.
This proves the lemma
\end{proof}
The existence of a point at which $C(a_c, a_{+}, a_{-})>0$ (Lemma \ref{lem:positive}) and a point at which $C(a_c, a_{+}, a_{-})<0$ (Lemma \ref{lem:negative}), together with the continuity of $C(a_c, a_{+}, a_{-})$ and the connectedness of the K\"ahler cone, imply by the intermediate value theorem that there exists a triple satisfying $\eqref{ineq:ample}$ for which $C(a_c, a_{+}, a_{-})=0$.
This completes the proof of Theorem \ref{thm:main}.

%Sec3.3
\subsection{Toric geometric viewpoint}\label{sec:toric}
In this subsection, we give a toric-geometric proof of Theorem \ref{thm:main}. 
As in Section \ref{sec:horospherical}, we mainly consider the family 
\[
Y_k=\P_{\P^{k+1}\times \P^{k-1}}(\O\oplus \O(-1,1)), \qquad k\geq 2.
\]
In the toric seeting, Hultgren's criterion provides a coupled analogue of Wang--Zhu's criterion for ordinary K\"ahler--Einstein metrics. 
More precisely, Hultgren proved that the existence of coupled K\"ahler--Einstein metrics is equivalent to a barycentre condition for a Minkowski decomposition of the anticanonical polytope. 
Let $P$ be an $n$-dimensional reflexive Delzant polytope, and let $X$ be the associated smooth toric Fano $n$-fold.
Suppose that a set of polytopes $P_1,\dots, \P_\ell \subset \R^n$ are polytopes satisfying the Minkowski decomposition
\[
P=P_1+\dots +P_\ell.
\]
For each $i$, let $\alpha_i$ denote the K\"ahler class corresponding to $P_i$. Then the following statements are equivalent 
(see \cite[Theorem 2]{H19}):
\begin{enumerate}
\item The decomposition $(\alpha_1, \dots \alpha_\ell)$ of $c_1(X)$ admits a coupled K\"ahler--Einstein tuple.
\item $\displaystyle
\frac{\int_{P_1}\bs x\, dV}{\Vol(P_1)}+\cdots +\frac{\int_{P_\ell}\bs x\, dV}{\Vol(P_\ell)}=0$.
\end{enumerate}

Set $p=k+1$, $q=k-1$ and $n:=p+q+1=2k+1$. Let $\Sigma_k$ be the corresponding complete fan of $Y_k$ in $N_\R$.
Recall from Section \ref{sec:arbitrary} that the primitive generators of the rays of $\Sigma_k$ are
\begin{align*}
\bs v_i&=\bs e_{i} \quad (1\leq i \leq p), \qquad \bs v_j=\bs e_{j} \quad (p+1\leq j \leq p+q),  \\
\bs v_+&=-\sum_{i=1}^p\bs e_i+\bs e_n, \qquad  \bs v_-=-\sum_{j=p+1}^{p+q}\bs e_j-\bs e_n, \qquad \bs v_0=-\bs e_n, \quad \bs v'_0=\bs e_n.
\end{align*}
The primitive relations of $\Sigma_k$ are
\[
\sum_{i=1}^p\bs v_i+\bs v_+=\bs v'_0, \qquad
\sum_{j=p+1}^{p+q}\bs v_j+\bs v_-=\bs v_0, \qquad
\bs v'_0+\bs v_0=0. 
\]
After applying a linear change of primitive generators, we have
\begin{align*}
\bs v_i=\bs e_i& \longmapsto \bs e_i-k\bs e_n=:\bs w_i \qquad \qquad (1\leq i \leq p), \qquad 
&\bs v'_0=\bs e_n \longmapsto -k(k+2)\bs e_n=: \bs w'_0, \\
\bs v_j=\bs e_j &\longmapsto \bs e_j+(k+2)\bs e_n=:\bs w_j \quad (p+1 \leq j \leq p+q),   \qquad
&~~~~~ \bs v_0=-\bs e_n \longmapsto k(k+2)\bs e_n=: \bs w_0.
\end{align*}
Then we observe that
\[
\bs v_+ \longmapsto \bs w_+:=-\sum_{i=1}^{k+1}\bs e_i -k\bs e_n, \qquad \qquad
\bs v_- \longmapsto \bs w_-:=-\sum_{i=k+2}^{2k}\bs e_j+(k+2)\bs e_n.
\]
For $c\in \left( \frac{1}{4}, \frac{3}{4}\right)$, define
\begin{align*}
P_c&:=\left \{ \bs x\in \R^n \, \bigg | \, \braket{\bs x, \bs w_i}\geq -\frac{1}{2} \,\,\,\,\, (1\leq i \leq k+1), \quad
\braket{\bs x, \bs w_j}\geq -c \,\,\,\,\, (k+2\leq j \leq 2k), \right . \\
& \left.  \hspace{2.3cm} \braket{\bs x, \bs w_{\pm}}\geq -\frac{1}{2}, \quad \braket{\bs x, \bs w_{0}}\geq -\frac{1}{2}, \quad 
\braket{\bs x, \bs w'_{0}}\geq -\frac{1}{2}  \right \}.
\end{align*}
Then $P:=P_c+P_{1-c}$ is the anticanonical polytope of $Y_k$.

Set $u:=k(k+2)x_n$. The generators $\bs w_0$, $\bs w'_0$ (corresponding to the fiber direction) give 
\[
-\frac{1}{2}\leq u \leq \frac{1}{2}.
\]
Fix one value of $u$. Then $x_n$ is fixed, so the remaining variables are $x_1, \dots, x_{k+1}$ and
$x_{k+2}, \dots, x_{2k}$. Let 
\[
P_c(u)=P_c\cap\set{u=\text{constant}} 
\]
denote the slice polytope. Consider the first block of $P_c(u)$ corresponding to the $\P^{k+1}$-factor of the base:
\begin{equation}\label{ineq:1st}
x_i\geq -\frac{1}{2}+\frac{u}{k+2}=-\frac{1}{2}+\alpha(u) \quad (1\leq i \leq k+1),
\qquad \sum_{i=1}^{k+1}x_i \leq \frac{1}{2}-\frac{u}{k+2}=\frac{1}{2}-\alpha(u),
\end{equation}
where $\alpha(u):=\frac{u}{k+2}$. Since $u$ is fixed, $\alpha(u)$ is constant.
Translating the variables $y_i:=x_i+\frac{1}{2}-\alpha(u)$, we find that
\begin{align*}
\eqref{ineq:1st} \quad \Leftrightarrow \quad y_i\geq 0 \quad (1\leq i \leq k+1), \qquad \sum_{i=1}^{k+1}y_i
\leq (k+2)\left( \frac{1}{2}-\alpha(u)\right)=:L_1(u).
\end{align*}
Therefore, the first factor of the slice polytope is the standard simplex
\[
\Set{(y_1, \dots, y_{k+1})\in \R^{k+1}| y_i\geq 0 \, \,\, (i=1, \dots, k+1), \quad  \sum_{i=1}^{k+1}y_i \leq L_1(u)},
\]
whose side length is 
\begin{equation*}\label{def:A(u)}
L_1(u)=(k+2)\left( \frac{1}{2}-\alpha(u)\right)=\frac{k+2}{2}-u=:A(u).
\end{equation*}
Similarly for the second block of $P_c(u)$,
\[
x_j\geq -c+\gamma(u), \qquad \sum_{j=k+2}^{2k}x_j\leq c-\gamma(u), \qquad \gamma(u):=-\frac{u}{k}.
\]
Translating the variables $z_j:=x_j+c-\gamma(u)$, we see that the slice is
\[
\Set{(z_{k+2}, \dots, z_{2k})\in \R^{k-1}| z_j\geq 0 \, \,\, (j=k+2, \dots, 2k), \quad  \sum_{j=k+2}^{2k}z_j \leq L_2(u)},
\]
whose side length is
\begin{equation*}\label{def:B(u)}
L_2(u)=k(c-\gamma(u))=u+kc=:B_c(u).
\end{equation*}
If we denote
\[
\D^m(L):=\Set{(y_1,\dots, y_m)\in \R^m | y_i\geq 0, \quad \sum_{i=1}^m y_i \leq L},
\]
then $\Vol(\D^m(L))=\frac{L^m}{m!}$. Therefore $\displaystyle P_c(u)=\D^{k+1}\bigl(A(u)\bigr)\times \D^{k-1}\bigl(B_c(u)\bigr)$,
and hence
\[ 
\Vol(P_c(u))=\frac{A(u)^{k+1}}{(k+1)!}\cdot \frac{B_c(u)^{k-1}}{(k-1)!}.
\]
By symmetry, all barycenter coordinates except the $x_n$-coordinate vanish.
Thus, the relevant barycenter coordinate is proportional to
\[
M_k(c)=\frac{\displaystyle \int_{-\frac{1}{2}}^{\frac{1}{2}}uA(u)^{k+1}B_c(u)^{k-1}du}
{\displaystyle \int_{-\frac{1}{2}}^{\frac{1}{2}}A(u)^{k+1}B_c(u)^{k-1}du}.
\]
By Hultgren's toric criterion for coupled KE metrics \cite[Theorem 2]{H19}, the pair $(P_c, P_{1-c})$ determines a two-coupled KE metric
if and only if there exists a constant $c\in\left(\frac{1}{4}, \frac{3}{4}\right)$ satisfying 
\[
M_k(c)+M_k(1-c)=0.
\]
Now define $F_k(c):=M_k(c)+M_k(1-c)$ which is a continuous function on $\left(\frac{1}{4}, \frac{3}{4}\right)$.
%Lem
\begin{lemma}\label{lem:Mk_negative}
We have $F_k\left(\frac{1}{2}\right)<0$ for every $k\geq 2$. In particular, $Y_k$ does not admit an ordinary K\"ahler-Einstein metric.
\end{lemma}
\begin{proof}
The second assertion follows immediately from the first, since $P=P_{\frac{1}{2}}+P_{\frac{1}{2}}=2P_{\frac{1}{2}}$.
In particular, $F_k\left(\frac{1}{2}\right)=2M_k\left(\frac{1}{2}\right)$. Thus, it suffices to prove $M_k\left(\frac{1}{2}\right)<0$ for all $k\geq 2$.

For $c=\frac{1}{2}$, 
\[
\frac{d}{du}\left( A(u)^{k+2}B_{\frac{1}{2}}(u)^k\right)=-2(k+1)uA(u)^{k+1}B_{\frac{1}{2}}(u)^{k-1}.
\]
Therefore,
\begin{align*}
\int_{-\frac{1}{2}}^{\frac{1}{2}}uA(u)^{k+1}B_{\frac{1}{2}}(u)^{k-1}du&=-\frac{1}{2(k+1)}\left[A(u)^{k+2}B_{\frac{1}{2}}(u)^{k} \right]^{u=\frac{1}{2}}_{u=-\frac{1}{2}}\\
&=-\frac{1}{2(k+1)}\left\{A\left(\frac{1}{2}\right)^{k+2}B_{\frac{1}{2}}\left(\frac{1}{2}\right)^{k}-
A\left(-\frac{1}{2}\right)^{k+2}B_{\frac{1}{2}}\left(-\frac{1}{2}\right)^{k}
\right\}.
\end{align*}
Since $A\left(\frac{1}{2}\right)=B_{\frac{1}{2}}\left(\frac{1}{2}\right)=\frac{k+1}{2}$, $A\left(-\frac{1}{2}\right)=\frac{k+3}{2}$
and $B_{\frac{1}{2}}\left(-\frac{1}{2}\right)=\frac{k-1}{2}$, it is enough to show that
\begin{equation}\label{ineq:k}
(k+1)^{2k+2}-(k+3)^{k+2}(k-1)^k>0
\end{equation}  
for proving the lemma. To prove this inequality, define
\[
R_k:=\frac{(k+3)^{k+2}(k-1)^k}{(k+1)^{2k+2}}, \qquad k\geq 2.
\]
We claim that $R_k<1$ for every $k\geq 2$.

Indeed, for $k=2$,
\[
R_2=\frac{5^4\cdot 1}{3^6}=\frac{625}{729}<1.
\]
For $k\geq 3$, we estimate the logarithm of $R_k$.
Since
\[
R_k=\left(\frac{k+3}{k+1}\right)^2\left(1+\frac{2}{k+1}\right)^k\left(1-\frac{2}{k+1}\right)^k=\left(\frac{k+3}{k+1}\right)^2\left(1-\frac{4}{(k+1)^2}\right)^k,
\]
we have
\[
\log R_k=2\log\left(1+\frac{2}{k+1}\right)+k\log\left(1-\frac{4}{(k+1)^2}\right).
\]
Using 
\[
\log(1+x)<x-\frac{x^2}{2}+\frac{x^3}{3}, \qquad \log(1-x)<-x-\frac{x^2}{2}, \qquad (x>0)
\]
we obtain
\begin{align*}
\log R_k&< 2\left( \frac{2}{k+1}-\frac{2}{(k+1)^2}+\frac{8}{3(k+1)^3} \right)-k\left( \frac{4}{(k+1)^2}+\frac{8}{(k+1)^4} \right)\\
&=\frac{8(2-k)}{3(k+1)}<0 \qquad (k\geq 3).
\end{align*}
Therefore, $\log R_k<0$ for all $k\geq 2$. This proves that the inequality $\eqref{ineq:k}$ holds for every $k\geq 2$. 
Since $F_k\left(\frac{1}{2}\right)=2M_k\left(\frac{1}{2}\right)<0$, this proves the lemma.
\end{proof}
As the K\"ahler range is $\left( \frac{1}{4}, \frac{3}{4}\right)$, it suffices to show the following lemma to prove Theorem \ref{thm:main}.
%Lem
\begin{lemma}\label{lem:Fk_positive}
For all sufficiently large $k$, we have $F_k\left(\frac{1}{4}\right)>0$.
\end{lemma}
\begin{proof}
For convenience, set $w_c(u):=A(u)^{k+1}B_c(u)^{k-1}$. Then 
\[
M_k(c)=\frac{\int_{-\frac{1}{2}}^{\frac{1}{2}}uw_c(u)\, du}{\int_{-\frac{1}{2}}^{\frac{1}{2}}w_c(u)\, du}.
\]
Setting $D_0:=\int_{-\frac{1}{2}}^{\frac{1}{2}}w_{c_0}(u)\, du$ and $D_1:=\int_{-\frac{1}{2}}^{\frac{1}{2}}w_{1-c_0}(u)\, du$, we have
\[
F_k(c_0)=\frac{1}{D_0D_1}\int_{-\frac{1}{2}}^{\frac{1}{2}}\int_{-\frac{1}{2}}^{\frac{1}{2}}(t+s)w_{c_0}(t)w_{1-c_0}(s)\,dtds=:I_k.
\]
Thus it is enough to show $I_k>0$ for all sufficiently large $k$ and $c_0=\frac{1}{4}$ to prove the lemma.
Substituting $c_0=\frac{1}{4}$, we have
\[
w_{c_0}(t)=\left( \frac{k+2}{2}-t\right)^{k+1}\left( \frac{k}{4}+t\right)^{k-1}, \qquad
w_{1-c_0}(s)=\left( \frac{k+2}{2}-s\right)^{k+1}\left( \frac{3k}{4}+s\right)^{k-1}.
\] 
Fix $k\in \Z_{>0}$, and let
\[
C_{1,k}:=\left( \frac{k}{2}\right)^{k+1}\left( \frac{k}{4}\right)^{k-1}, \qquad
C_{2,k}:=\left( \frac{k}{2}\right)^{k+1}\left( \frac{3k}{4}\right)^{k-1}.
\]
Factoring out the terms independent of $t$ and $s$, we obtain
\[
\frac{w_{c_0}(t)}{C_{1,k}}=\left( 1+\frac{2(1-t)}{k}\right)^{k+1}\left( 1+\frac{4t}{k}\right)^{k-1}, \qquad
\frac{w_{1-c_0}(s)}{C_{2,k}}=\left( 1+\frac{2(1-s)}{k}\right)^{k+1}\left( 1+\frac{4s}{3k}\right)^{k-1}.
\]
Using $\log \left(1+\frac{x}{k}\right)=\frac{x}{k}+O(k^{-2})$, %uniformly on $\left[ -\frac{1}{2}, \frac{1}{2}\right]$, 
we obtain
\begin{align*}
\log \left(\frac{w_{c_0}(t)}{C_{1,k}}\right)&=(k+1)\log\left( 1+\frac{2(1-t)}{k}\right)+(k-1)\log \left( 1+\frac{4t}{k}\right)\\
&=(k+1)\left( \frac{2(1-t)}{k}+O(k^{-2})\right)+(k-1) \left( \frac{4t}{k}+O(k^{-2}) \right)=2(1+t)+O(k^{-1}),
\end{align*}
uniformly in $t$.
Hence, we conclude
\[
\frac{w_{c_0}(t)}{C_{1,k}}=e^{2(1+t)}(1+O(k^{-1})), \qquad \frac{w_{1-c_0}(s)}{C_{2,k}}=e^{2-\frac{2}{3}s}(1+O(k^{-1}))
\]
as well. Consequently,
\begin{align}
\begin{split}\label{lim:Ik}
\lim_{k\to \infty}\frac{I_k}{C_{1,k}C_{2,k}}&=e^4\int_{-\frac{1}{2}}^{\frac{1}{2}}\int_{-\frac{1}{2}}^{\frac{1}{2}}(t+s)e^{2t-\frac{2}{3}s}\, dtds \\
&=e^4\left\{ 
\left( \int_{-\frac{1}{2}}^{\frac{1}{2}}te^{2t}\,dt \right)\left( \int_{-\frac{1}{2}}^{\frac{1}{2}}e^{-\frac{2}{3}s}\,ds \right)
+\left( \int_{-\frac{1}{2}}^{\frac{1}{2}}e^{2t}\,dt \right)\left( \int_{-\frac{1}{2}}^{\frac{1}{2}}se^{-\frac{2}{3}s}\,ds \right)
\right\}.
\end{split}
\end{align}
To evaluate the sign of $\eqref{lim:Ik}$, we define
\[
Z(a)=\int_{-\frac{1}{2}}^{\frac{1}{2}} e^{ax}\,dx.
\]
Define the probability density $\displaystyle p_a(x)=\frac{e^{ax}}{Z(a)}$,~~ $\displaystyle \left(-\frac{1}{2}\leq x\leq \frac{1}{2}\right)$ which is normalized since
\[
 \int_{-\frac{1}{2}}^{\frac{1}{2}}p_a(x)\,dx=\frac{1}{Z(a)}\int_{-\frac{1}{2}}^{\frac{1}{2}} e^{ax}\,dx=1.
\]
The expectation of $x$ with respect to this density is
\[
m(a):= \int_{-\frac{1}{2}}^{\frac{1}{2}}xp_a(x)\,dx=\frac{1}{Z(a)}\int_{-\frac{1}{2}}^{\frac{1}{2}} xe^{ax}\,dx=\frac{Z'(a)}{Z(a)}.
\]
Indeed,
\[
Z'(a)=\frac{d}{da} \int_{-\frac{1}{2}}^{\frac{1}{2}} e^{ax}\, dx=\int_{-\frac{1}{2}}^{\frac{1}{2}} xe^{ax}\, dx.
\]
Then the double integral in $\eqref{lim:Ik}$ is equal to
\begin{align*}
Z(2)Z\left( -\frac{2}{3}\right)\left[\frac{\int_{-\frac{1}{2}}^{\frac{1}{2}} te^{2t}\, dt}{Z(2)}
+\frac{\int_{-\frac{1}{2}}^{\frac{1}{2}} se^{-\frac{2}{3}s}\, ds}{Z\left(-\frac{2}{3} \right)}
\right]&=
Z(2)Z\left( -\frac{2}{3}\right)\left(m(2)+m\left(-\frac{2}{3} \right) \right).
\end{align*}
Since $Z(2)$, $Z\left( -\frac{2}{3}\right)$ are positive values, the sign of $\eqref{lim:Ik}$ is exactly the sign $m(2)+m\left(-\frac{2}{3} \right)$.

For $a\neq 0$,
\[
Z(a)=\int_{-\frac{1}{2}}^{\frac{1}{2}} e^{ax}\,dx=\frac{e^{\frac{a}{2}}-e^{-\frac{a}{2}}}{a}=\frac{2\sinh\left(\frac{a}{2}\right)}{a}.
\]
Hence, $\log Z(a)=\log 2+\log \sinh\left(\frac{a}{2}\right)-\log a$. 
Since $m(a)=\frac{Z'(a)}{Z(a)}=\frac{d}{da}\log Z(a)$, we obtain
\begin{align*}
m(a)&=\frac{\cosh\left(\frac{a}{2}\right)}{2\sinh\left(\frac{a}{2}\right)}-\frac{1}{a}
=\frac{1}{2}\coth \left(\frac{a}{2}\right)-\frac{1}{a}.
\end{align*}
Using this, we have
\begin{align}
m(2)+m\left( -\frac{2}{3}\right)&=1+\frac{1}{2}\left( \coth 1-\coth\left( \frac{1}{3}\right) \right) \notag\\
&=1+\frac{1}{2}\left( \frac{-\sinh \left( \frac{2}{3}\right)}{\sinh \left( \frac{1}{3}\right)\sinh 1 } \right)
=1-\frac{\cosh \left(\frac{1}{3} \right)}{\sinh 1}, \label{eq:MV1}
\end{align}
where we used $\coth x-\coth y=\dfrac{\sinh(y-x)}{\sinh x \sinh y}$ in the above equality. 
Since $\displaystyle \sinh 1>\cosh  \left( \frac{1}{3}\right)$, we conclude from $\eqref{eq:MV1}$ that
\[
m(2)+m\left( -\frac{2}{3}\right)>0.
\]
%Moreover,
%\[
%\sinh 1=2\sinh \left( \frac{1}{2}\right) \cosh  \left( \frac{1}{2}\right) > \cosh  \left( \frac{1}{2}\right) >\cosh  \left( \frac{1}{3}\right).
%\]
%Thus, $m(2)+m\left( -\frac{2}{3}\right)>0$ by $\eqref{eq:MV1}$.
Consequently, $I_k>0$ for all sufficiently large $k$.

Since $D_0D_1>0$, it follows that 
\[
F_k\left( \frac{1}{4}\right)>0
\]
for all sufficiently large $k$. This completes the proof of lemma.
\end{proof}
By Lemmas \ref{lem:Mk_negative}, \ref{lem:Fk_positive} together with the intermediate value theorem, there exists $c_k\in \left(\frac{1}{4}, \frac{3}{4} \right)$ such that $F_k(c_k)=0$. Since $P_{c_k}$ and $P_{1-c_k}$ determine the K\"ahler classes
\[
\alpha_1, \alpha_2\in H^{1,1}(Y_k,\R), \qquad \alpha_1+ \alpha_2=c_1(Y_k),
\]
the pair $(\alpha_1,\alpha_2)$ admits a two-coupled KE metric. This completes the proof of Theorem \ref{thm:main}.

%Sec4
\section{Nonexistence of coupled K\"ahler-Einstein metrics on toric Fano threefold $\mathcal F_2$.}\label{sec:ToriFano3}
In this section we study the toric Fano threefold $\mathcal F_2$ the unique toric Fano threefold that does not admit a K\"ahler-Einstein metric but is not excluded a priori from admitting a coupled K\"ahler-Einstein (cKE) metric (cf. Proposition \ref{prop:Aut}). 
More precisely, we investigate whether there exists an $\ell$-tuple decomposition
\[
c_1(X)=\alpha_1+\dots +\alpha_\ell
\]
whose associated K\"ahler classes admit an $\ell$-coupled KE metric.

Our main result in this section shows that this possibility does not occur. More precisely, for every $\ell\in \Z_{>0}$, we prove that there exists no decomposition $(\alpha_1, \dots, \alpha_\ell)$ of $c_1(\mathcal F_2)$ admitting a cKE metric , when considering one-parameter family of ample divisors.
%Prop
\begin{proposition}\label{prop:F2}
Let $\pi:\mathcal F_2 \to \P^1$ denote the projection onto $\P^1$.
Let $\beta_1, \beta_2\in H^{1,1}(\mathcal F_2)$ be the K\"ahler classes corresponding to the torus-invariant divisors 
\[
D_1:=\pi^{-1}(0), \quad D_2:=\pi^{-1}(\infty),
\]
respectively. Consider a decomposition of $c_1(\mathcal F_2)$ of the form $(\alpha_1, \alpha_2)$, where 
\[
\alpha_1=\frac{1}{2}\left(c_1(\mathcal F_2)-\beta_1-\beta_2 \right)+C(\beta_1+\beta_2), \qquad 
\alpha_2=\frac{1}{2}\left(c_1(\mathcal F_2)-\beta_1-\beta_2 \right)+(1-C)(\beta_1+\beta_2).
\]
Here the constants $C$ and  $1-C$ lie in the interval $\left( \frac{1}{4}, \frac{3}{4}\right)$ and are chosen so that both $\alpha_1$ and $\alpha_2$ define K\"ahler classes.
Then the decomposition does not admit a cKE metric for any $C\in \left( \frac{1}{4}, \frac{3}{4}\right)$.
\end{proposition}
\begin{proof}
The moment polytope $P\subset M_\R$ associated with $(\mathcal F_2, \mathcal O(-K_{\mathcal F_2}))$ is given by
\begin{align*}
P&=\left \{\bs x=(x_1,x_2, x_3)\in \R^3\, \big |  \,  -x_2+x_3 \geq -1, ~ x_2 \geq -1, ~ x_1 \geq -1,  \right. \\
& \qquad \,  \hspace{7.5cm} \left.  -1 \leq x_1-x_3 \leq 1, ~ -1\leq x_3 \leq 1  \right \}.
\end{align*}
A direct computation shows that %the barycenter and the volume of $P$ are 
\begin{align}\label{eq:Bary1} 
\bs b_P&:=\int_P\bs x\, dV=\left( \frac{5}{12},  \frac{5}{12}, \frac{5}{6} \right)=\frac{5}{12}\bigl(1,1,2 \bigr),
\qquad \text{and} \qquad
\Vol(P):=\int_P  \bs 1\, dV =6, %\notag
\end{align}
respectively. In view of \eqref{eq:Bary1}, we introduce the linear transformation
\begin{equation}\label{eq:LinTrans}
A={
\begin{matrix} \\
 \\
 \\
 \end{matrix}
 }^{t}\!\!\!\left(\begin{array}{ccc}1 & 0 &  -1 \\  
  0 &1  & -1    \\  
  0 &0  & -2  
   \end{array}\right).
\end{equation}
Let $\Sigma$ denote the complete fan in $N_\R$ corresponding to
 $\mathcal F_2$. The one-dimensional cones
 \[
 \Sigma(1)=\Set{\R_{\geq 0}\bs v_1, \dots, \R_{\geq 0}\bs v_8}
 \]
 are generated by primitive vectors $\bs v_1,\dots, \bs v_8 \in N$, which correspond to the torus invariant divisors $D_1, \dots, D_8\subset \mathcal F_2$. Specifically, these primitive vectors are:
\begin{align*}
\bs v_1&=\begin{pmatrix}
0 \\ -1 \\ 1
\end{pmatrix},~\quad
\bs v_2=\begin{pmatrix}
0 \\ 1 \\ 0
\end{pmatrix},~\quad
\bs v_3=\begin{pmatrix}
1 \\ 0 \\ 0
\end{pmatrix},~\quad
\bs v_4=\begin{pmatrix}
-1 \\ 0 \\ 0
\end{pmatrix},~ \\
\bs v_5&=\begin{pmatrix}
1 \\ 0 \\ -1
\end{pmatrix},~\quad
\bs v_6=\begin{pmatrix}
-1 \\ 0 \\ 1
\end{pmatrix},~\quad
\bs v_7=\begin{pmatrix}
0 \\ 0 \\ 1
\end{pmatrix},~\quad
\bs v_8=\begin{pmatrix}
0 \\ 0 \\ -1
\end{pmatrix}.
\end{align*}
It is known that there exist exactly two distinct $3$-dimensional Fano polytopes $Q\subset N_\R$ with eight vertices, which arise as the polar dual polytopes of the moment polytopes $P\subset M_\R$; see \cite[Proposition $2.5.8$]{Bat99}). For toric Fano threefolds of type $\mathcal F$, the primitive collections are
\[
\Set{\bs v_1, \bs v_2}, ~~\quad  \Set{\bs v_3, \bs v_4, \bs v_5, \bs v_6, \bs v_7, \bs v_8}.
\]
The distinction between the geometric structure of the K\"ahler-Einstein (KE) manifold 
\[
\mathcal F_1= \P^1\times \mathrm{Bl}_{(3\text{pts})}(\P^2)
\] 
and non KE manifold $\mathcal F_2$ is reflected in the primitive relations
\[
%\begin{cases}
\bs v_1+ \bs v_2 = 0  \quad \text{for}\quad \mathcal F_1, \qquad 
\bs v_1+ \bs v_2 =\bs v_7  \quad \text{for}\quad \mathcal F_2.
%\end{cases}
\]
Motivated by this distinction, we single out the pair of torus-invariant divisors $\set{D_1, D_2}\subset \mathcal F_2$ corresponding to the primitive collection $\Set{\bs v_1, \bs v_2}$.

Applying the linear transformation $A$ in $\eqref{eq:LinTrans}$ to the generators $\bs v_1, \dots, \bs v_8$, we obtain new generators
\begin{align*}
\bs v_1'&=\begin{pmatrix}
0 \\ -1 \\ -1
\end{pmatrix},~\quad
\bs v_2'=\begin{pmatrix}
0 \\ 1 \\ -1
\end{pmatrix},~\quad
\bs v_3'=\begin{pmatrix}
1 \\ 0 \\ -1
\end{pmatrix},~\quad
\bs v_4'=\begin{pmatrix}
-1 \\ 0 \\ 1
\end{pmatrix},~ \\
\bs v_5'&=\begin{pmatrix}
1 \\ 0 \\ 1
\end{pmatrix},~\quad
\bs v_6'=\begin{pmatrix}
-1 \\ 0 \\ -1
\end{pmatrix},~\quad
\bs v_7'=\begin{pmatrix}
0 \\ 0 \\ -2
\end{pmatrix},~\quad
\bs v_8'=\begin{pmatrix}
0 \\ 0 \\ 2
\end{pmatrix}.
\end{align*}
Then the resulting polytope
\[
P':=\Set{\bs x\in \R^3 | \bracket{\bs x, \bs v_i'}\geq -1, ~~ i=1,2, \dots, 8}
\]
satisfies 
\[
\bs b_{P'}=\left( 0,0, -\frac{5}{24}
\right), \qquad %\text{and} \quad
\Vol(P')=\frac{1}{2}\cdot \Vol(P)=3.
\]
Next, we introduce a one-parameter family of polytopes
\begin{equation}\label{eq:NewPoly} 
P'(C):=\Set{\bs x\in \R^3 | \braket{\bs x, \bs v_i'} \geq -C, ~\, i=1,2, \quad  \braket{\bs x, \bs v_i'} \geq -\frac{1}{2}, ~\, i\neq 1,2}.
\end{equation}
One readily verifies that for any $C\in \left(\frac{1}{4}, \frac{3}{4} \right)$, $P'$ and $P'(C)$ have  the same number of vertices and are combinatorially equivalent. In particular, the divisors
\[
D(C):=C(D_1+D_2)+\sum_{i\neq 1,2}\frac{1}{2}D_i \quad \text{and} \quad D(1-C):=(1-C)(D_1+D_2)+\sum_{i\neq 1,2}\frac{1}{2}D_i
\]
are ample for all $C\in  \left(\frac{1}{4}, \frac{3}{4} \right)$.
%Claim1
\begin{claim}\label{claim:Bary1}
\[
\int_{P'(C)}x_1\, dV=\int_{P'(C)}x_2\, dV=0.
\]
\end{claim}
\begin{proof}[Proof of Claim \ref{claim:Bary1}]
The sets $\Set{\bs v_1', \bs v_2'}$, $\Set{\bs v_3', \dots,  \bs v_8'}$ are each invariant under the linear transformation
\begin{equation*}\label{eq:LinTrans2}
B=
\left(\begin{array}{ccc}
-1 & 0 &  0 \\  
  0 &-1  & 0    \\  
  0 &0  & 1  
   \end{array}\right)
\in \SL(3,\Z). 
\end{equation*}
Consequently, the polytope $P'(C)$ is invariant under $B$, and hence its 
unnormalized barycenter $\bs b_{P'(C)}$ is fixed by $B$. 
The fixed-point set of $B$ is precisely the $x_3$-axis.
Therefore, the first and second coordinates of the barycenter vanish, which implies
\[
\int_{P'(C)}x_1\, dV=\int_{P'(C)}x_2\, dV=0.
\]
\end{proof}

%Claim2
\begin{claim}\label{claim:Bary2}
\[
\int_{P'(C)}x_3\, dV=-\frac{5}{384}
\]
\end{claim}
which is independent of the parameter $C\in  \left(\frac{1}{4}, \frac{3}{4} \right)$.
\begin{proof}[Proof of Claim \ref{claim:Bary2}]
Let $D$ denote the two-dimensional square 
\[
D:=\Set{(x_1,x_3)\in \R^2 | -\frac{1}{2} \leq x_1-x_3 \leq \frac{1}{2}, ~~ -\frac{1}{2} \leq x_1+x_3 \leq \frac{1}{2}.
}
\]
By definition,
\begin{align*}
P'(C)&=\left \{\bs x=(x_1,x_2,x_3)\in \R^3 ~ \bigg | ~ -x_2-x_3 \geq -C,~ x_2-x_3\geq -C,  \right. \\ 
& \hspace{3.5cm} \left.  -\frac{1}{2}\leq x_1-x_3 \leq \frac{1}{2},~  -\frac{1}{2}\leq x_1+x_3 \leq \frac{1}{2}, ~ -\frac{1}{2}\leq 2x_3 \leq \frac{1}{2}  \right \}.
\end{align*}
Hence a point $\bs x=(x_1,x_2,x_3)$ belongs to $P'(C)$ if and only if 
\[
x_3 \in \left(-\frac{1}{4} , \frac{1}{4}  \right),\qquad  |x_2|\leq C-x_3, ~ \qquad (x_1, x_3)\in D,
\]
provided that $C\in \left( \frac{1}{4}, \frac{3}{4}\right)$.
Therefore, by direct computation, 
\begin{align*}
\int_{P'(C)}x_3\, dV&=\int_{x_3=-\frac{1}{4}}^{\frac{1}{4}}\int_{x_1=-\frac{1}{2}+|x_3|}^{\frac{1}{2}-|x_3|} \int_{x_2=x_3-C}^{C-x_3}x_3 \,dx_2\, dx_1\, dx_3\\
&=\int_{x_3=-\frac{1}{4}}^{\frac{1}{4}}\int_{x_1=-\frac{1}{2}+|x_3|}^{\frac{1}{2}-|x_3|}2x_3(C-x_3)dx_1\, dx_3\\
&=2\int_{-\frac{1}{4}}^{\frac{1}{4}}x_3(C-x_3)(1-2|x_3|)dx_3.
\end{align*}
Splitting the integral at $x_3=0$, we obtain
\begin{align*}
\int_{P'(C)}x_3\, dV&=
2\left(\int_{-\frac{1}{4}}^0x_3(C-x_3)(1+2x_3)dx_3+
\int_0^{\frac{1}{4}}x_3(C-x_3)(1-2x_3)dx_3\right)\\
&=\left(-\frac{C}{24}-\frac{5}{768} \right)+\left(\frac{C}{24}-\frac{5}{768} \right)=-\frac{5}{384}.
\end{align*}
This completes the proof.
\end{proof}
We emphasize that, by Claim \ref{claim:Bary2}, the integral $\int_{P'(C)}x_3\, dV$ is independent of the parameter $C$. A similar direct computation shows that
\begin{align*}
\Vol(P'(C)):=\int_{P'(C)}\bs 1\, dV&=\int_{x_3=-\frac{1}{4}}^{\frac{1}{4}}\int_{x_1=-\frac{1}{2}+|x_3|}^{\frac{1}{2}-|x_3|} \int_{x_2=x_3-C}^{C-x_3}\bs 1 \,dx_2\, dx_1\, dx_3\\
&=2\int_{-\frac{1}{4}}^{\frac{1}{4}}(C-x_3)(1-2|x_3|)dx_3\\
&=2\left(
\int_{-\frac{1}{4}}^0(C-x_3)(1+2x_3)dx_3+\int_0^{\frac{1}{4}}(C-x_3)(1-2x_3)dx_3\right)\\
&=\frac{1}{24}\left(9C+1+9C-1\ \right)=\frac{3}{4}C.
\end{align*}
By Hultgren's criterion (\cite[Theorem 2]{H19}), a necessary and sufficient condition for $(\alpha_1, \alpha_2)$ to admit a cKE metric is 
\begin{equation}\label{eq:cKE}
\frac{\int_{P'(C)}x_3\, dV}{\Vol(P'(C))}+\frac{\int_{P'(1-C)}x_3\, dV}{\Vol(P'(1-C))}=0, \qquad
C\in  \left(\frac{1}{4}, ~\frac{3}{4} \right).
\end{equation}
Substituting the above computations into $\eqref{eq:cKE}$, we have
\begin{align}\label{eq:cKE2}
\frac{4}{3C}\cdot \left(-\frac{5}{384} \right)+\frac{4}{3(1-C)}\cdot \left(-\frac{5}{384} \right)=0
\quad \Leftrightarrow \quad
\frac{1}{C}+\frac{1}{1-C}=0.
\end{align}
Clearly, $\eqref{eq:cKE2}$ has no solution for $C\in  \left(\frac{1}{4}, \frac{3}{4} \right)$.
Therefore, no cKE metric exists for any $C\in  \left(\frac{1}{4}, \frac{3}{4} \right)$.
\end{proof}

%Cor
\begin{corollary}\label{cor:F2}
Let $\mathcal F_2$, $D_i$ and $\beta_i$ $(i=1,2)$ be as in Proposition \ref{prop:F2}.
Let $\ell\geq 1$.
Consider a decomposition of $c_1(\mathcal F_2)$ given by an $\ell$-tuple 
$(\alpha_1,\dots, \alpha_\ell)$, where each $\alpha_i$ is defined by
\begin{align}\label{eq:alphas}
\alpha_j&=\frac{1}{\ell}\bigl(c_1(\mathcal F_2)-\beta_1-\beta_2 \bigr)+C_j(\beta_1+\beta_2), \qquad
j=1, \dots, \ell
\end{align}
with constants $C_j$ satisfying
\[
 \sum_{j=1}^\ell C_j=1, \qquad C_j\in\left(\frac{1}{2\ell}, \frac{2\ell-1}{2\ell} \right).
\]
These conditions ensure that $\alpha_1, \dots, \alpha_\ell$ are K\"ahler classes.
Then there is no coupled K\"ahler-Einstein (cKE) metric for any such $\ell$-tuple decomposition of $c_1(\mathcal F_2)$.
\end{corollary}
\begin{proof}
We first illustrate the argument for $\ell=3$. 
The corresponding polytope defined in $\eqref{eq:NewPoly}$ is
\begin{align*}
P'(C)&=\left \{\bs x\in \R^3 \,|\, x_2+x_3\leq C, -x_2+x_3\leq C, -\frac{1}{3}\leq x_1-x_3 \leq \frac{1}{3}, \right. \\
 &\qquad \hspace{6.5cm} \left. -\frac{1}{3}\leq x_1+x_3 \leq \frac{1}{3}, -\frac{1}{3}\leq 2x_3 \leq \frac{1}{3} \right \}.
\end{align*}
Repeating the computation in the proof of Claim \ref{claim:Bary2} gives
\begin{align}
\begin{split}\label{eq:Bary3couple}
\int_{P'(C)}x_3\, dV&=\int_{x_3=-\frac{1}{6}}^{\frac{1}{6}}\int_{x_1=-\frac{1}{3}+|x_3|}^{\frac{1}{3}-|x_3|} \int_{x_2=x_3-C}^{C-x_3}x_3 \,dx_2\, dx_1\, dx_3 
=-\frac{5}{1944}, \\
\Vol(P'(C))&=\int_{x_3=-\frac{1}{6}}^{\frac{1}{6}}\int_{x_1=-\frac{1}{3}+|x_3|}^{\frac{1}{3}-|x_3|} \int_{x_2=x_3-C}^{C-x_3}\bs 1 \,dx_2\, dx_1\, dx_3 =\frac{C}{3}.
\end{split}
\end{align}
Therefore, a necessary and sufficient condition for $(\alpha_1, \alpha_2, \alpha_3)$  to admit a cKE metric is
\begin{equation}\label{eq:cKE3couple}
\frac{\int_{P'(C_1)}x_3\,dV}{\Vol(P'(C_1))}+\frac{\int_{P'(C_2)}x_3\,dV}{\Vol(P'(C_2))}+\frac{\int_{P'(1-C_1-C_2)}x_3\,dV}{\Vol(P'(1-C_1-C_2))}=0
\end{equation}
with 
\[
C_1,~ C_2, ~ 1-C_1-C_2\in \left( \frac{1}{6}, \frac {5}{6}\right).
\]
Substituting $\eqref{eq:Bary3couple}$ into $\eqref{eq:cKE3couple}$, we obtain
\begin{align*}
\eqref{eq:cKE3couple} \quad &\Leftrightarrow \quad
\frac{1}{C_1}+\frac{1}{C_2}+\frac{1}{1-C_1-C_2}=0
\end{align*}
or equivalently,
\[
C_1C_2(1-C_1-C_2)\bigl( C_1+C_2-C_1^2-C_2^2-C_1C_2\bigr )=0.
\]
\begin{claim}\label{claim:RealSol}
There is no real solution of
\[
C_1+C_2-C_1^2-C_2^2-C_1C_2=0
\]
satisfying ${1}/{6}< C_1, ~ C_2, ~ 1-C_1-C_2<{5}/{6}$.
\end{claim}
\begin{proof}[Proof of Claim \ref{claim:RealSol}]
Set $x:=C_1$ and $y:=C_2$, and consider
\begin{equation}\label{eq:3coupled}
x^2+y^2+xy=x+y.
\end{equation}
Let $s:=x+y$. Since $x^2+y^2+xy=(x+y)^2-xy$, $\eqref{eq:3coupled}$ is equivalent to
\begin{equation}\label{eq:3coupled2}
(x+y)^2-xy=x+y \quad \Leftrightarrow \quad xy=s(s-1).
\end{equation}
Our K\"ahler condition implies 
\[
\frac{1}{6}< x<\frac{5}{6}, \quad \frac{1}{6}< y<\frac{5}{6}, \quad \frac{1}{6}< 1-x-y<\frac{5}{6}. \quad
\]
The last inequality gives
\[
\frac{1}{6}<1-s <\frac{5}{6} \quad \Leftrightarrow \quad \frac{1}{6} <s <\frac{5}{6}.
\]
Together with $x,y>1/6$, we obtain $s=x+y>1/3$. Hence,
\begin{equation}\label{ineq:condi}
\frac{1}{3}<s<\frac{5}{6}.
\end{equation}
From $\eqref{ineq:condi}$, we have $s>0$ and $s-1<0$, and therefore $s(s-1)<0$.

On the other hand, $\eqref{eq:3coupled2}$ gives 
\[
xy=s(s-1)<0,
\]
contradicting the assumption $x,y>\frac{1}{6}$.

Consequently, there is no real solution, satisfying
\[
x+y-x^2-y^2-xy=0
\]
together with
\[
\frac{1}{6}<x <\frac{5}{6},\quad \frac{1}{6}<y <\frac{5}{6},\quad \frac{1}{6}<1-x-y <\frac{5}{6}.
\]
\end{proof}

Now let $\ell\in \Z_{>0}$ be arbitrary. For each $j=1, \dots, \ell$, consider the polytope 
\begin{align*}
P'(C_j)&:=\left \{\bs x=(x,y,z)\in \R^3 \,~\bigg|~\, y+z\leq C_j, ~ -y+z\leq C_j, ~ -\frac{1}{\ell}\leq x-z \leq \frac{1}{\ell}, \right. \\
 &\qquad \hspace{6.5cm} \left. -\frac{1}{\ell}\leq x+z \leq \frac{1}{\ell}, ~ -\frac{1}{\ell}\leq 2z \leq \frac{1}{\ell} \right \},
\end{align*}
where $C_\ell:=1-\sum_{j=1}^{\ell-1}C_j$.
A direct computation shows that the barycenter of $P'(C_j)$ is independent of the parameter $C_j\in \left( \frac{1}{2\ell}, \frac{2\ell-1}{2\ell}\right)$, i.e.,
\begin{align*}
\int_{P'(C_j)}z\, dV&=
\int_{z=-\frac{1}{2\ell}}^{\frac{1}{2\ell}}\int_{x=-\frac{1}{\ell}+|z|}^{\frac{1}{\ell}-|z|} \int_{y=z-C_j}^{C_j-z}z \,dy\, dx\, dz =-R_1,
\end{align*}
where $R_1\in \R$ is a constant independent of the parameter $C_j$. 
Observe that the $\ell$-tuple $(\alpha_1, \dots, \alpha_\ell)$ given in \eqref{eq:alphas} determines
the K\"ahler classes as long as $C_j\in \left( \frac{1}{2\ell}, \frac{2\ell-1}{2\ell}\right)$ for each $1\leq j \leq \ell$. Moreover, $\Vol(P'(C_j))$ can be written as
\[
\Vol(P'(C_j))=\int_{P'(C_j)}\bs 1\, dV=R_2C_j,
\]
for some positive constant $R_2\in \R_{>0}$. Hence, $(\alpha_1, \dots, \alpha_\ell)$ admits a cKE metric if and only if
\begin{equation}\label{eq:k-coupled}
\sum_{j=1}^\ell \frac{\int_{P'(C_{j})}z\,dV}{\Vol(P'(C_{j}))}=0, \qquad C_j\in \left(\frac{1}{2\ell}, \frac{2\ell-1}{2\ell} \right).
\end{equation}
Substituting each value of $\int_{P'(C_j)}z\,dV$ and $\Vol(P'(C_{j}))$ into $\eqref{eq:k-coupled}$, we obtain
\begin{equation}\label{eq:k-coupled2}
-\sum_{j=1}^\ell \frac{-R_1}{R_2C_j}=0
\quad \Leftrightarrow \quad
\sum_{j=1}^\ell \prod_{s \neq j} C_s=0.
\end{equation}
However, since $C_j>0$ for all $j$, every term in the right hand side of $\eqref{eq:k-coupled2}$
is strictly positive, yielding a contradiction.
Therefore, equation $\eqref{eq:k-coupled2}$ has no solution satisfying
\[
C_j \in \left(\frac{1}{2\ell}, \frac{2\ell-1}{2\ell}  \right),
\]
and the proof is completed.
\end{proof}

%Rem
\begin{remark}\rm
Among toric Fano fourfolds $D_{19}$, $H_{10}$, $J_2$ and $Q_{17}$ in Proposition \ref{prop:Aut}, the author also proved that
$D_{19}$ is the unique non-KE toric Fano manifold for which there exists a two-coupled decomposition $(\alpha_1, \alpha_2)$ of the first Chern class $c_1(X)$ that admits a cKE metric, when considering the natural choice of torus-invariant divisors generating the Picard group $\Pic(X)$. However, proving the nonexistence of cKE metrics in general is more subtle, since one needs to
parameterize {\emph{all}} possible choices of torus-invariant divisors generating $\Pic(X)$.
\end{remark}

%%%%%%%%%%%%%%%%%%%%%%%%%%%%%%%%

\end{document}